\documentclass[11pt]{article}

\usepackage[a4paper,margin=1in]{geometry}
\usepackage{amsmath,amssymb,amsthm}
\usepackage{enumitem}
\usepackage{authblk}
\usepackage[hidelinks]{hyperref}

\newtheorem{theorem}{Theorem}

\newcommand{\bb}[2]{\binom{#1}{#2}}

\title{A 920-block explicit construction guaranteeing a triple intersection with every $6$-subset of $[60]$}
\author{Paulo Henrique Cunha Gomes\\ School of Technology (FT) – University of Campinas (Unicamp)\\Limeira,13484-332,São Paulo,Brazil\\ e-mail:p072049@dac.unicamp.br}
\date{}

\begin{document}

\maketitle

\begin{abstract}
We present an explicit family $\mathcal{B}$ of $920$ subsets of size $6$ of $[60]=\{1,\dots,60\}$
with the property that every $6$-subset $S\subset[60]$ intersects at least one block $B\in\mathcal{B}$
in at least three elements, i.e.\ $|S\cap B|\ge 3$.
The construction is purely combinatorial, based on a partition of the ground set into pairs and a
pigeonhole argument. We also record a simple counting lower bound and discuss how different
partitions of the ten base blocks affect the emergence of triple intersections.
\end{abstract}

\noindent\textbf{Keywords:} covering designs; packing and covering; block designs; set systems; Johnson graph; explicit construction; pigeonhole principle.

\noindent\textbf{2020 Mathematics Subject Classification:} Primary 05B40; Secondary 05B05.

\section{Introduction}
Let $[n]=\{1,2,\dots,n\}$. Fix $n=60$ and $k=6$. We study families
$\mathcal{B}\subset \bb{[60]}{6}$ satisfying
\[
\forall S\in \bb{[60]}{6}\ \exists B\in \mathcal{B}\ \text{ such that }\ |S\cap B|\ge 3.
\]
Equivalently, $\mathcal{B}$ is a covering of the Johnson graph $J(60,6)$ with covering radius $3$
(since the Johnson distance between two $6$-sets equals $6-|S\cap B|$).

\section{Construction}
Partition the ground set into two disjoint halves
\[
A=\{1,\dots,30\},\qquad C=\{31,\dots,60\}.
\]
Partition each half into $15$ disjoint pairs, explicitly:
\[
P_i=\{2i-1,2i\}\subset A\quad (1\le i\le 15),
\qquad
Q_j=\{30+2j-1,30+2j\}\subset C\quad (1\le j\le 15).
\]
Define the two families
\[
\mathcal{B}_A=\{P_i\cup P_j\cup P_k:1\le i<j<k\le 15\},
\qquad
\mathcal{B}_C=\{Q_i\cup Q_j\cup Q_k:1\le i<j<k\le 15\}.
\]
Each block in $\mathcal{B}_A$ and $\mathcal{B}_C$ has size $6$, and
\[
|\mathcal{B}_A|=|\mathcal{B}_C|=\bb{15}{3}=455.
\]
Finally, add the ten base blocks of the fixed partition of $[60]$ into consecutive $6$-sets:
\[
G_1=\{1,\dots,6\},\ G_2=\{7,\dots,12\},\ \dots,\ G_{10}=\{55,\dots,60\}.
\]
Let
\[
\mathcal{B}=\mathcal{B}_A\ \cup\ \mathcal{B}_C\ \cup\ \{G_1,\dots,G_{10}\}.
\]
Then $|\mathcal{B}|=455+455+10=920$.

\section{Main result}
\begin{theorem}\label{thm:main}
Let $\mathcal{B}$ be the family of $920$ blocks constructed above.
For every $S\subset[60]$ with $|S|=6$, there exists $B\in\mathcal{B}$ such that
$|S\cap B|\ge 3$.
\end{theorem}

\begin{proof}
Let $S\subset[60]$ with $|S|=6$, and write $S_A=S\cap A$ and $S_C=S\cap C$.
Since $|S_A|+|S_C|=6$, either $|S_A|\ge 3$ or $|S_C|\ge 3$.

Assume $|S_A|\ge 3$ (the case $|S_C|\ge 3$ is identical with $Q$-pairs).
Choose distinct elements $a,b,c\in S_A$.
Each of $a,b,c$ belongs to a unique pair among $\{P_1,\dots,P_{15}\}$; denote these pairs by
$P_{i_a},P_{i_b},P_{i_c}$.

If $i_a,i_b,i_c$ are distinct, then
$B=P_{i_a}\cup P_{i_b}\cup P_{i_c}\in\mathcal{B}_A$ contains $\{a,b,c\}$.

If two of the indices coincide, say $P_{i_a}=P_{i_b}$, choose any $r\notin\{i_a,i_c\}$ and set
$B=P_{i_a}\cup P_{i_c}\cup P_r\in\mathcal{B}_A$.
Again $B$ contains $\{a,b,c\}$.
Thus in all cases $|S\cap B|\ge 3$.
\end{proof}

\section{Bounds and minimality}
Let $M$ denote the minimum possible size of a family $\mathcal{F}\subset\bb{[60]}{6}$ satisfying
the triple-intersection property.

A single block $B$ covers exactly those $6$-subsets $S$ with $|S\cap B|\ge 3$; their number is
\[
N=\sum_{i=3}^{6}\bb{6}{i}\bb{54}{6-i}=517{,}870.
\]
Since $\bb{60}{6}=50{,}063{,}860$, any such family must have at least
\[
\left\lceil \frac{\bb{60}{6}}{N}\right\rceil
=
\left\lceil \frac{50{,}063{,}860}{517{,}870}\right\rceil
=97
\]
blocks. Therefore,
\[
97 \le M \le 920.
\]
Determining $M$ exactly (or improving the explicit upper bound significantly) remains open here.

\section{A (5,5) partition viewpoint}
The ten base blocks $G_1,\dots,G_{10}$ form a partition of $[60]$ into two groups of five blocks:
$A=G_1\cup\cdots\cup G_5$ and $C=G_6\cup\cdots\cup G_{10}$.
For any $6$-subset $S\subset[60]$, let $x=|S\cap A|$ and $y=|S\cap C|$.
Then $x+y=6$, hence $\max\{x,y\}\ge 3$, so at least one half contains a triple of $S$.

Moreover, each base block $G_i$ is internally partitioned into three disjoint pairs, so each half
yields $15$ disjoint pairs. The triple inside that half lies in at most three of these pairs; selecting
those pairs and taking their union yields a block in $\mathcal{B}_A$ or $\mathcal{B}_C$ intersecting
$S$ in at least three points, which is precisely the mechanism behind the construction.

\section{A (3,3,4) partition and an obstruction}
If instead the ten base blocks are partitioned into three groups of sizes $(3,3,4)$ and a $6$-subset
$S$ satisfies the balanced distribution $(2,2,2)$ across the three groups, then no group contains a
triple of $S$. In particular, no argument restricted to recombining pairs \emph{within a single group}
can force a block meeting $S$ in three elements. This shows that, unlike the $(5,5)$ split, the
$(3,3,4)$ partition does not necessarily force the existence of a triple intersection.

\section{Conclusion}
We provided a simple explicit family of $920$ blocks of size $6$ guaranteeing a triple intersection
with every $6$-subset of $[60]$. The construction is elementary and deterministic, and the additional
discussion highlights how the underlying block partition structure governs when triple intersections
are forced.

\end{document}